\documentclass[amsfonts]{amsart}

\usepackage[pdftex]{hyperref}

\usepackage{mathpazo}
\usepackage{textcomp}
\usepackage{nicefrac} 
\usepackage{xfrac}    
\usepackage{frcursive}
\usepackage[T1]{fontenc}

\usepackage{amsmath}
\usepackage{amsfonts}
\usepackage{amssymb}
\usepackage{makerobust}
\MakeRobustCommand\overrightarrow
 
\usepackage{mathtools}

\DeclarePairedDelimiter\floor{\lfloor}{\rfloor}

\usepackage[all,arc]{xy}
\usepackage{enumerate}
\usepackage{mathrsfs}
\usepackage[utf8]{inputenc}
\usepackage{soul}
\usepackage{hyperref} 

\newcommand{\Rr}{\mathbb{R}}
\newcommand{\Nn}{\mathbb{N}}
\newcommand{\Cc}{\mathbb{C}}
\newcommand{\Zz}{\mathbb{Z}}

\newtheorem{thm}{Theorem}[section]

\newtheorem{prop}[thm]{Proposition}
\newtheorem{defi}[thm]{Definition}
\newtheorem{rema}[thm]{Remark}
\newtheorem{coro}[thm]{Corollary}
\newtheorem{lem}[thm]{Lemma}

\newtheorem{convention}[thm]{Convention}

\title{Symplectic resolutions of the quotient of $ {\mathbb R}^2 $ by a non-finite symplectic group}

\author{Hichem Lassoued$^{1}$, Camille Laurent-Gengoux$^{2}$}

\address{\emph{1}, Nantes University, 44000 Nantes, France.}

\address{\emph{2}, Institut Elie Cartan de Lorraine, I.E.C.L., UMR 7502, Universit\'e de Lorraine, rue Augustin Fresnel, 57000 Metz, France.  }

\date{January 3rd, 2022}

\begin{document}

\begin{abstract}
We construct smooth symplectic resolutions of the quotient of $\mathbb R^2 $ under some \emph{infinite} discrete sub-group of ${\mathrm{ GL}}_2(\mathbb R) $ preserving a log-symplectic structure. This extends from algebraic geometry to smooth real differential geometry the Du Val symplectic resolution of $\mathbb C^2 \hspace*{-1.5pt} / \hspace*{-1.5pt}G$, with $G \subset  {\mathrm{ SL}}_2(\mathbb C) $ a finite group. 
The first of these infinite groups is $G=\mathbb Z$, identified to triangular matrices with spectrum $\{1\} $. Smooth functions on the quotient $\Rr^2 \hspace*{-1.5pt} / \hspace*{-1.5pt} G $ come with a natural Poisson bracket, and $\Rr^2\hspace*{-1.5pt} / \hspace*{-1.5pt}G$ is for an arbitrary $k \geq 1$ set-isomorphic to the real Du Val singular variety $A_{2k} = \{(x,y,z) \in \mathbb R^3 , x^2 +y^2= z^{2k}\}$. We show that each one of the usual minimal resolutions of these Du Val varieties are symplectic resolutions of $\Rr^2\hspace*{-1.5pt} / \hspace*{-1.5pt}G$. The same holds for $G'=\mathbb Z \rtimes \mathbb Z\hspace*{-1.5pt} / \hspace*{-1.5pt}2\mathbb Z$ (identified to triangular matrices with spectrum $\{\pm 1\} $), with the upper half of $D_{2k+1} $ playing the role of $A_{2k}$.
\vspace{.5cm}

\noindent
\end{abstract}

\maketitle

\tableofcontents

\section{Introduction}
This article aims to initiate the study of quotients of smooth Poisson manifolds $(M,\pi) $ by an \underline{infinite and non-proper} discrete group $G$ of Poisson automorphisms of $(M,\pi)$. 
In the sequel, we call \emph{quotient of $M$ by $G$} and denote by $ M\hspace*{-1.5pt} / \hspace*{-1.5pt}G$ (with an abuse of notation) the quotient of $M$ under the closure of the equivalence relation induced by $G$, \emph{i.e.} the equivalence relation generated by the relation $ x \sim y $ if and only if there exists sequences $ (x_n)_{n \in \mathbb N},(y_n)_{n \in \mathbb N}$ with limits $x$ and $y$ respectively such that $ x_n$ and $y_n$ are in the same $G$-orbit for all $n \in \mathbb N$. Such quotients are not varieties 
but they are always "Poisson sets" \cite{Hichem} in the following sense: a real valued function on $M\hspace*{-1.5pt} / \hspace*{-1.5pt}G $ is said to be \emph{smooth} if its pull-back to $M$ is smooth on $M$. In other words, we define the algebra of smooth functions on $M\hspace*{-1.5pt} / \hspace*{-1.5pt}G $ by  $ C^\infty(M\hspace*{-1.5pt} / \hspace*{-1.5pt}G) := C^\infty(M)^G$. Since $G$ preserves the Poisson structure $\pi $, the algebra $ C^\infty(M\hspace*{-1.5pt} / \hspace*{-1.5pt}G)$ comes equipped with a Poisson bracket (see \cite{c3}, Section 5.4).

\begin{defi} \cite{Hichem1} \label{def:Lassoued}
Let $(M,\pi_M)$ be a Poisson manifold, and $C^\infty(M\hspace*{-1.5pt} / \hspace*{-1.5pt}G) := C^\infty(M)^G$. 
A \emph{symplectic realization} of $M\hspace*{-1.5pt} / \hspace*{-1.5pt}G$ is a triplet $(\Sigma, \pi_{\Sigma}, \varphi)$ where $(\Sigma, \pi_{\Sigma})$ is a symplectic manifold (maybe with boundary), $\varphi: \Sigma \to M\hspace*{-1.5pt} / \hspace*{-1.5pt}G$ is a map such that
\begin{enumerate}
    \item[i)] $\varphi \colon \Sigma \to M\hspace*{-1.5pt} / \hspace*{-1.5pt}G $ is smooth, \emph{i.e.} $\varphi^*(F)$ is a smooth function on $\Sigma $  for all $F \in C^\infty(M\hspace*{-1.5pt} / \hspace*{-1.5pt}G)$,
    \item[ii)] the algebra morphism $\varphi^*\colon C^\infty(M\hspace*{-1.5pt} / \hspace*{-1.5pt}G) \to C^{\infty}(\Sigma)$ is a Poisson morphism.
\end{enumerate}
Such a symplectic realization $(\Sigma,\pi_\Sigma,\varphi) $ is called a \emph{symplectic resolution} of $M\hspace*{-1.5pt} / \hspace*{-1.5pt}G $ when:
\begin{enumerate}
    \item[iii)] there exists an open dense $G$-saturated subset $U \subset M$ on which the $G$-action is discrete, proper and free\footnote{This implies that $U\hspace*{-1.5pt} / \hspace*{-1.5pt}G $ is a manifold (and makes sense of Condition \emph{(iiii)}).},
    \item[iiii)] $ \varphi : \varphi^{-1}(U\hspace*{-1.5pt} / \hspace*{-1.5pt}G) \to U\hspace*{-1.5pt} / \hspace*{-1.5pt}G $ is a diffeomorphism.
\end{enumerate}
\end{defi}

\noindent
Symplectic realizations of $M\hspace*{-1.5pt} / \hspace*{-1.5pt}G $ defined as in
Definition \ref{def:Lassoued} imitates the nowadays classical notion of symplectic realizations of smooth Poisson manifolds \cite{BX16}-\cite{DCW}-\cite{DZ}-\cite{Karazev}.
Symplectic realizations of $M\hspace*{-1.5pt} / \hspace*{-1.5pt}G $ defined as in
Definition \ref{def:Lassoued} imitates the nowadays classical notion of
symplectic resolutions of affine Poisson varieties as in \cite{sym}-\cite{BaoHua}-\cite{c1234}: indeed, it is an extension of this notion to the context of real differential geometry (as in \cite{c1234}).

\noindent
In algebraic geometry, examples of symplectic resolutions  include quotients of symplectic vector spaces by \textit{finite} subgroup of the symplectic group, see \cite{BellamySchedler}. The most basic examples of such quotients are the very classical Du Val singularities $ (A_n, D_n,E_6,E_7,E_8)$, see \cite{Du Val}.
It is known that the minimal desingularization of these singularities  is a symplectic resolution \cite{sym}-\cite{res}. This implies that those (more precisely, their real parts) are also symplectic resolutions in the sense of Definition~\ref{def:Lassoued}. 

\noindent
We wish to investigate quotients of Poisson vector spaces which are symplectic (at regular points of the action) under infinite discrete groups and their symplectic resolutions.
We will see that for $G$ a discrete infinite subgroup of $GL(V) $, there may be very few non-constant algebraic $G$-invariant functions, while $G$-invariant smooth functions $ C^{\infty}(V \hspace*{-1.5pt} / \hspace*{-1.5pt} G ) $ may be still large enough to be worth of study. Our problem can not be addressed within algebraic geometry: We have to work with smooth functions.

\noindent
But this is certainly not an easy matter. Our present purpose is to understand in detail two examples that we think to be "generic", at least for b-manifolds. Both will turn out to be surprisingly subtle, although they are simply the quotients of $\mathbb R^2$ by the groups 
 \begin{equation}\label{eq:GG'}
 G = \left\{ \left(  \begin{array}{cc} 1 & n\\ 0 & 1 \end{array}
 \right) \,  \middle| n \in \mathbb Z \right\} \hbox{ and }  G' = \left\{ \left(  \begin{array}{cc} 1 & n\\ 0 & \epsilon \end{array}
 \right) \,  \middle|  \epsilon \in \{\pm1\}, n \in \mathbb Z   \right\} \end{equation}
which are isomorphic to $\Zz $ and $\Zz\hspace*{-1.5pt} / \hspace*{-1.5pt}2\Zz \rtimes \Zz$ respectively.
These infinite but discrete groups act on $\mathbb R^2 $ by Poisson automorphisms of the   Poisson structure
 \begin{equation}\label{eq:Poisson1} \pi_{\mathbb R^2} := \frac{1}{2 \pi} \, q  \, \frac{\partial}{ \partial p} \wedge \frac{\partial}{ \partial q}  \end{equation}
where $(p,q) $ stands for the canonical coordinates on $\mathbb R^2 $.  This Poisson bracket is log-symplectic, making $\mathbb R^2 $ a b-manifold \cite{Miranda}.

\noindent
It was proven in \cite{Hichem1} that the Poisson structure \eqref{eq:Poisson1} does not admit symplectic resolutions.  In contrast, the conclusion of the present article is that the quotients $\mathbb R^2 \hspace*{-1.5pt} / \hspace*{-1.5pt}G$ and $\mathbb R^2 \hspace*{-1.5pt} / \hspace*{-1.5pt}G'$ do admit symplectic resolutions. Although these are only two examples, the present article claims to be a first step toward a general theory of symplectic resolutions for singular Poisson spaces. In subsequent works, examples in higher dimension shall be given, and the pattern that we see here will repeat itself: Poisson singular spaces only make sense in the smooth world, they are equipped with one-to-one smooth maps onto some affine Poisson spaces that admit symplectic resolutions. Also, unlike in the algebraic geometry case, proving these results will require classical tools from analysis (e.g. Fourier series) used in a non-trivial manner. Last, our symplectic resolutions are not unique, which is very different from the Du Val case. 



\vspace{.5cm}
 
\noindent Let us describe the content of the article. In Section \ref{sec:quotient}, we study the algebras $C^{\infty}(\Rr^2\hspace*{-1.5pt} / \hspace*{-1.5pt}G):= C^{\infty}(\Rr^2)^G$ and  $C^{\infty}(\Rr^2\hspace*{-1.5pt} / \hspace*{-1.5pt}G'):= C^{\infty}(\Rr^2)^{G'}$.
In the process, we will explain why, unlike in the finite group case, real analytic or holomorphic contexts are not relevant here because they admit too few invariant functions. 
For smooth invariant functions, we give a useful Fourier series expansion. 

\noindent Section \ref{sec:DuVal} is dedicated to reminders on Du Val's symplectic resolutions (in the real case) of $ A_{2k}$ and $ D_{2k+1}$. In Section \ref{sec:Quotient2}, we show that for all $k \geq 2$, the quotient space $\Rr^2\hspace*{-1.5pt} / \hspace*{-1.5pt}G $ is in bijection (as a set) with the quotient singularity of Du Val $$ A_{2k}=\left\{(x,y,z)\in\Rr^3\mid x^2+y^2=z^{2k}\right\}.$$
We then show that this bijection $\underline{\phi_k}$ behaves well with respect to Poisson structures. By "behaves well", we mean that $\underline{\phi_k}^*:C^\infty(\mathbb R^3)  \to C^{\floor{k/2}} (\mathbb R^2\hspace*{-1.5pt} / \hspace*{-1.5pt}G)$ is a Poisson algebra morphism.


\noindent We also show similar results for the second group $G' \simeq  \Zz \rtimes \Zz\hspace*{-1.5pt} / \hspace*{-1.5pt}2\Zz$, but it is then the Du Val singularities of type $ D_{2k+1}=\left\{(x,y,z)\in\Rr^3\mid zx^2+y^2=z^{2k+1}\right\} $ which intervenes (more precisely, $D_{2k+1} \cap \{ z \geq 0 \}$). 


\noindent 
 Section \ref{sec:resolution} concludes this construction with an extremely surprising result: the (usual) symplectic resolution $\varphi:Z_k \to A_{2k} $ of the Du Val singularity $A_{2k}$ induces a symplectic resolution of the Poisson quotient $\Rr^2\hspace*{-1.5pt} / \hspace*{-1.5pt}G $, when composed with $\underline{\phi_k}^{-1}: A_{2k} \to \Rr^2\hspace*{-1.5pt} / \hspace*{-1.5pt}G   $. In other words, $ \underline{\phi_k}^{-1} \circ \varphi : Z_k \to  \Rr^2\hspace*{-1.5pt} / \hspace*{-1.5pt}G $ satisfies the requirements of Definition \ref{def:Lassoued}. This is not an easy result, and we have to deploy various techniques to achieve this, including Fourier analysis, and subtle considerations on the Du Val resolution. 

\noindent We also show similar results for the second group $G':=\Zz \rtimes \Zz\hspace*{-1.5pt} / \hspace*{-1.5pt}2\Zz$.
The usual symplectic resolution $\varphi:Z_k \to D_{2k+1} $, restricted to $\varphi^{-1} ( D_{2k+1} \cap \{z \geq 0\})  $, composed with $\overline{\phi_k}^{-1} :  D_{2k+1} \cap \{z \geq 0\}) \to \mathbb R^2\hspace*{-1.5pt} / \hspace*{-1.5pt}G' $  gives a smooth symplectic resolution of the Poisson quotient $\Rr^2\hspace*{-1.5pt} / \hspace*{-1.5pt}G'$ in this case.

\section{Quotient of $\mathbb{R}^2$ by two infinite groups}\label{sec:quotient}

We equip $ \Rr^2 $ with the coordinates $(p,q)$. Consider the infinite groups $G \simeq \Zz$ and  $G' \simeq \Zz\rtimes\Zz\hspace*{-1.5pt} / \hspace*{-1.5pt}2\Zz$ as in (\ref{eq:GG'}) and their actions on $\Rr^2$ by 

\begin{equation*}
\begin{array}{ccccc}
 & & \Zz\times\Rr^2 & \to & \Rr^2 \\
 & & n \cdot (p,q) & \mapsto & (p+nq,q) 
\end{array} \hbox{ \hspace{.5cm} and } \begin{array}{ccccc}
 & & (\Zz\rtimes\Zz\hspace*{-1.5pt} / \hspace*{-1.5pt}2\Zz) \times\Rr^2 & \to & \Rr^2\\
 & & (n,\epsilon) \cdot (p,q) & \mapsto & (p+nq,\epsilon q)
\end{array}
\end{equation*}
for all $(n,\epsilon)\in\Zz\times\{-1,1\}$ and $(p,q) \in \Rr^2$. 

\noindent
The closure of the equivalence relation $\sim $ induced by the action of $G = \Zz$ above is given by 
\begin{equation}\label{hna} (p,q) \sim (p+nq,q),
\end{equation}
for any pair $(p,q) \in \Rr^{2}$, $n\in\mathbb{Z}$ and $(p,0) \sim (p',0)$ for all $p,p' \in \Rr$. It happens to be an equivalence relation which is closed \emph{i.e.}, $\left\{(p_1,q_1)\times(p_2,q_2)\in\Rr^2 \times \Rr^2  \mid (p_1,q_1)\sim(p_2,q_2)\right\}$ is a closed set of $\Rr^2\times\Rr^2$.
For $G'=\Zz\rtimes\Zz\hspace*{-1.5pt} / \hspace*{-1.5pt}2\Zz$, we still have to identify $(p,q)$ and $(p,-q)$.
We denote by $ \sim'$ this relation. 

\begin{convention}
We denote (with a slight abuse of notation) the equivalence classes of the closed relations $\sim $ and $\sim' $ by $ \mathbb R^2\hspace*{-1.5pt} / \hspace*{-1.5pt}G$ and $\Rr^2\hspace*{-1.5pt} / \hspace*{-1.5pt}G'$.
\end{convention}

In general, for $\sim$ a closed equivalence relation on a manifold $M$, we call \emph{smooth function, real analytic function or function of class $C^k $ on  $M\hspace*{-1.5pt} / \hspace*{-1.5pt}\sim$} a function whose pull-back on $M $ through the natural projection $M \to M\hspace*{-1.5pt} / \hspace*{-1.5pt}\sim $ is  smooth, respectively real analytic or of class $C^k $. Equivalently, functions of one of these given types on $M\hspace*{-1.5pt} / \hspace*{-1.5pt}\sim$ are functions on $M$ constant on the equivalences classes. 

When the equivalence relation is generated by the closure of a relation given by the action of a group $G$, smooth functions (resp. real analytic or $C^k$-functions) on the quotient are exactly $G$-invariant smooth functions (resp. real analytic or $C^k$-functions) on $M$.
In the case we are interested with, we therefore have $\mathcal{F}[\Rr^2]^G\simeq \mathcal{F}[\Rr^2\hspace*{-1.5pt} / \hspace*{-1.5pt}G]$ and $\mathcal{F}[\Rr^2]^{G'}\simeq \mathcal{F}[\Rr^2\hspace*{-1.5pt} / \hspace*{-1.5pt}G']$ for $\mathcal F = C^\infty$ or $\mathcal C^k $. 


Let us study this quotient. 
The algebra of invariant functions $\mathcal{F}[\Rr^2]^G$ and $\mathcal{F}[\Rr^2]^{G'}$ may not be finitely generated. We take $G= \Zz $ and we start with the algebra $C^{\infty}(\Rr^2\hspace*{-1.5pt} / \hspace*{-1.5pt}G)= C^\infty (\Rr^2)^{G}$.

\begin{lem}\label{didii}
Every function $f \in C^{\infty}(\Rr^2)^G $ satisfies for all $i \geqslant 1, j \geqslant 0$ 
   $$\frac{\partial^{i+j}f}{\partial p^i\partial q^j}(p,0)=0, \hspace{1cm} \forall p\in \Rr .$$
\end{lem}
\begin{proof}[Proof]
The $ k^{th} $ derivative of the relation $f(p,q)=f(p+nq,q)$ with respect to $q$ gives
\begin{equation*}
    \frac{\partial^kf}{\partial q^k}(p,q)=\sum_{i=0}^{k} {{k}\choose{i}} n^{k-i}\frac{\partial^kf}{\partial p^i\partial q^{k-i}}(p+nq,q).\end{equation*}
For $q=0$, we have $$\sum_{i=0}^{k} {{k}\choose{i}} n^{i}\frac{\partial^kf}{\partial p^i\partial q^{k-i}}(p,0)=\frac{\partial^kf}{\partial q^k}(p,0).$$
The right hand term coincides with the $i=0$ term on left hand. We therefore obtain for all $n \in \mathbb Z$
 $$\sum_{i=1}^{k} {{k}\choose{i}} n^{i}\frac{\partial^kf}{\partial p^i\partial q^{k-i}}(p,0)=0.$$
The coefficients of a polynomial function that admits infinity many roots are zero. In particular, the coefficients of the polynomial
 $$P_{k} (X):=\sum_{i=1}^{k}  {{k}\choose{i}} \frac{\partial^kf}{\partial p^i\partial q^{k-i}}(p,0) \, X^{i}$$
are zero for all $k \in {\mathbb N}^*$ and $ p \in {\mathbb R}$. This
completes the proof.
\end{proof}


\begin{prop}\label{honi}
The space $C^{\infty}(\Rr^2\hspace*{-1.5pt} / \hspace*{-1.5pt}G)$ of smooth functions on $\Rr^2\hspace*{-1.5pt} / \hspace*{-1.5pt}G$ decomposes as follows $$C^{\infty}(\Rr^2\hspace*{-1.5pt} / \hspace*{-1.5pt}G)=C^{\infty}(q)+\widetilde{C},$$ 
where $C^{\infty}(q)$ is the algebra of smooth functions which depends only on the variable $q$, and $\widetilde{C} \subset C^{\infty}(\Rr^2\hspace*{-1.5pt} / \hspace*{-1.5pt}G)$ is the subalgebra of $G$-invariant smooth functions vanishing with all their partial derivatives\footnote{\label{footnote1}I.e. smooth functions $F$ on $ \mathbb R^2$ that satisfy $\frac{\partial^{i+j} F}{\partial p^i \partial q^{j}} (p,0) =0$ for all $p \in \mathbb R, i,j \in \mathbb N_0$.} along the line $q=0$. 
\end{prop}

\begin{proof}[Proof]
%
\noindent Consider a $G$-invariant smooth function $f(p,q)$. Define smooth functions $g,h$ on  $\Rr $ and $\Rr^2 $ respectively by $h(q):=g(0,q)$ and $g(p,q):=f(p,q)-h(q)$. We have $ f=h+g$ by construction. It therefore remains to show that $ g $ is an element of $\widetilde{C}$.

\noindent
Since the function $g$ is obviously $G$-invariant, Lemma \ref{didii} implies that for all $i \geq 1, j \geq 0$
$$\frac{\partial^{i+j} g}{\partial p^i\partial q^j }(p,0)=0, \ \forall p\in \Rr.$$
By integration,
$$\frac{\partial^j g}{\partial q^j}(p,0)=\frac{\partial^j g}{\partial q^j}(0,0)+\int_{t=0}^p\frac{\partial^{1+j}g}{\partial p\partial q^j}(t,0)dt.$$ 
Since $g(0,q)=0$ by construction, we have $\frac{\partial^jg}{\partial q^j}(0,0)=0$, and the first term on the right vanishes. 
Lemma \ref{didii} implies the vanishing of the second one. This completes the proof. 
\end{proof}

There is a similar result for the $G'$-action on $ \mathbb R^2$.
\begin{prop}\label{prop:honi2}
The space $C^{\infty}(\Rr^2\hspace*{-1.5pt} / \hspace*{-1.5pt}G')$ of smooth functions on $\Rr^2\hspace*{-1.5pt} / \hspace*{-1.5pt}G'$ decomposes as follows $$C^{\infty}(\Rr^2\hspace*{-1.5pt} / \hspace*{-1.5pt}G)=C^{\infty}(q^2)+\widetilde{C},$$ 
where $C^{\infty}(q^2)$ is the algebra of smooth even functions depending only on the variable $q^2$, and $\widetilde{C'} \subset C^{\infty}(\Rr^2\hspace*{-1.5pt} / \hspace*{-1.5pt}G')$ is the subalgebra of $G'$-invariant smooth functions vanishing with all their partial derivatives along the line $q=0$. 
\end{prop}

\noindent The following corollary shows that the real analytic case is not interesting, since there are too few $G$-invariant functions.

\begin{coro}
Real analytic functions on  $\mathbb R^2\hspace*{-1.5pt} / \hspace*{-1.5pt}G $  (on  $\mathbb R^2\hspace*{-1.5pt} / \hspace*{-1.5pt}G' $) are functions (resp. even function) which depend only on the variable $q$. 
\end{coro}
\begin{proof}[Proof]
A real analytic function on $ \Rr^2 $ which vanishes along a line with all its partial derivatives is zero on $ \Rr^2 $. 
The result now follows from Propositions \ref{honi} and \ref{prop:honi2}.
\end{proof}

\noindent This corollary also means that it is not possible to make such a study in the holomorphic context and study  $\Cc^2\hspace*{-1.5pt} / \hspace*{-1.5pt}G$ instead. There are not enough $G$-invariant holomorphic functions. 
Let us complete our description of smooth functions on  $\Rr^{2}\hspace*{-1.5pt} / \hspace*{-1.5pt}G$:

\begin{prop}\label{refr}
Any function $ f \in C^{\infty}(\Rr^2)^G \simeq C^{\infty}(\Rr^2\hspace*{-1.5pt} / \hspace*{-1.5pt}G) $ reads 
\begin{equation}\label{seriesf}
f(p,q)=\alpha_{0}(q)+\sum_{n\geq 1}\alpha_{n}(q)\cos\left(2n\pi\frac{p}{q}\right)+\sum_{n\geq1}\beta_{n}(q)\sin\left(2n\pi\frac{p}{q}\right),
\end{equation}
where $ \alpha_0$ is a smooth function on $\mathbb R $ and where   $\left(\alpha_{n}(q),\beta_{n}(q)\right)_{n\geq1}$ are smooth functions on $\mathbb R $ that vanish together with all their derivatives at $ q = 0$. 
Both above series, moreover, are absolutely convergent on any open ball. 
\end{prop}


%

\begin{proof}[Proof of Proposition \ref{refr}]
\noindent {\it{Step 1.}} Let us construct functions $\alpha_n, \beta_n $ that satisfy (\ref{seriesf}).  
For $q\in \Rr^*$ fixed, the function $p \mapsto f(p,q)$ is smooth and $q$-periodic. Its Fourier transform with respect to the variable $ p $ reads 
$$f(p,q)=\alpha_{0}(q)+\sum_{n\geq 1}\alpha_{n}(q)\cos\left(2n\pi\frac{p}{q}\right)+\sum_{n\geq1}\beta_{n}(q)\sin\left(2n\pi\frac{p}{q}\right) $$
\noindent Let us use the change of variables $u=\frac{p}{q}$ to find a friendly expression of the Fourier coefficients:
$$
\begin{array}{rcll}
\alpha_{0}(q)&=& \frac{1}{q}\int_{-q\hspace*{-1.5pt} / \hspace*{-1.5pt}2}^{q\hspace*{-1.5pt} / \hspace*{-1.5pt}2}f(p,q)dp,  & \\
& =& \int_{-1/2}^{1/2}f(uq,q)du & (*_0), \\
& \hbox{and} & & \\
\alpha_n(q)&=&\frac{2}{q}\int_{-q/2}^{q/2}f(p,q)\cos\left(2n\pi\frac{p}{q}\right)dp,& \\
&=& 2 \int_{-1/2}^{1/2}f(uq,q)\cos\left(2n\pi u \right)du, &  (*_n) \\
& \hbox{and} & & \\ \beta_n(q)&=&\frac{2}{q}\int_{-q/2}^{q/2}f(p,q)\sin\left(2n\pi\frac{p}{q}\right)dp,& \\
 &=&  2\int_{-1/2}^{1/2}f(uq,q)\sin\left(2n\pi u \right)du.& (**_n)
\end{array}
$$
Formulas $(*_0,*_n,**_n) $ still make sense for $q=0$, so that the functions $\alpha_n, \beta_n $ are defined on $ \mathbb R$. They are  smooth as integrals of smooth functions on a compact interval.

\noindent {\it{Step 2.}}  Let us show that for all $ n \geq 1 $, the Fourier coefficients $\alpha_n(q),\beta_n(q)\in C^{\infty}(\Rr,\Rr)$ are smooth functions which are zero with all their derivatives at $q=0$.
\begin{enumerate}
    \item[A.] If the function $f(p,q)$ only depends on $q$, then $\alpha_n=\beta_n=0$ for all $n \geqslant 1$, while $\alpha_0(q)=f(p,q) $ and the result is obvious. 
    \item[B.] If $f(p,q) $ belongs to $ \tilde{C}$, it vanishes with all its partial derivative along the line $q=0 $. A simple computation using ($*_n,**_n$) then gives the result.  
\end{enumerate} 
By Proposition \ref{honi}, any function in $C^\infty(\Rr^2\hspace*{-1.5pt} / \hspace*{-1.5pt}G) $ is a sum of a function of the  type A with a function of the second type B. The result therefore holds.

\noindent {\it{Step 3.}} Let us show that the series in (\ref{seriesf}) are absolutely convergent on any relatively compact subset $\mathcal U$ of $\Rr^2 $.  Let us integrate by part the right hand side of ($*_n$)  twice. Since $u \mapsto f(uq,q) $ is periodic of period $1$, all boundary term disappear, and we obtain
\begin{equation}\label{eq:alphaqk}  \alpha_n(q) = \frac{-1}{n^2} \frac{q^2}{4 \pi^2} \int_{-1/2}^{1/2} \frac{\partial^2 f(uq,q)}{ \partial q^2} {\rm cos}(2n \pi u) du. \end{equation}
Let  $K= \frac{B^2}{2 \pi^2} N$, where $B,N$ are the greatest possible value of $|q|$ and $|\frac{\partial^2 f(p,q)}{ \partial q^2}|$ on $\mathcal U $. We have $|\alpha_n(q)| < \tfrac{K}{n^2} $. The same applies to $(**_n) $.
Since the series $ \sum 1/n^2$ converges, both series in \eqref{seriesf} are absolutely convergent. \end{proof} 
 
In fact, the inequality \eqref{eq:alphaqk} can be made much stronger - and it will be needed later.

\begin{lem}
\label{lem:alphank}
For every $i \geq 1$, $k \geq 1 $ and $M >0$,  there exists a contant $K $ that does not depend on $n$ such that:
$$ \left|  \frac{\alpha_n^{(i)}(q)}{q^k}  \right|  \leq \frac{K}{n^k} \hbox{ and } \left|  \frac{\beta_n^{(i)}(q)}{q^k}  \right|  \leq \frac{K}{n^k}  $$
for all $-M<q<M $
\end{lem}
\begin{proof}
 Using successive integrations by part, we obtain
\begin{equation}\label{eq:alphaqk2}  \alpha_n^{(i)}(q) = \frac{1}{n^k} \frac{q^k}{(2 \pi)^k} \int_{-1/2}^{1/2} \frac{\partial^{k+i} f(uq,q)}{ \partial q^{k+i}} {\mathrm cossin_k}(2n \pi u) du. \end{equation}
where $ {\rm cossin_k}$ is ${\mathrm{ cos}},{\mathrm{ sin}}, -{\mathrm{ cos}}, -{\mathrm{ sin}} $ for $k \equiv 0,1,2,3  $ modulo $4$ respectively. Since $|{\rm cossin_k} | \leq 1$,
$$ K = \frac{ 1 }{(2\pi)^k} {\mathrm{Max}} \left\{ \left|  \frac{\partial^{k+i} f(uq,q)}{ \partial q^{k+i}}\right|  , {(u,q) \in [0,1] \times [-M,M]} \right\} $$
This maximum exists by compactness. This concludes the proof for the functions $\alpha_n^{(i)} $. The proof is similar for the functions $\beta_n^{(i)} $.
\end{proof}

Here is an immediate consequence of Lemma \ref{lem:alphank}:

\begin{prop}
Both series in Equation \eqref{seriesf} are uniformly convergent on any compact set, together with all their partial derivatives.
\end{prop}

\noindent For the group action $ G'$, a similar argument leads to:

\begin{prop}\label{refr1}
A function $f\in C^{\infty}(\Rr^{2}\hspace*{-1.5pt} / \hspace*{-1.5pt}G')$ reads $$f(p,q)=\alpha_{0}(q)+\sum_{n\geq 1}\alpha_{n}(q)\cos\left(2n\pi\frac{p}{q}\right)+\sum_{n\geq1} \beta_{n}(q)\sin\left(2n\pi\frac{p}{q}\right),$$
where the functions $\left(\alpha_{n},\beta_{n}\right)_{n\geq1}$ are smooth and vanish with all their derivatives at $ q = 0 $. Moreover the functions $ \alpha_{n} $ are even functions and the functions $ \beta_{n} $ are odd functions for all $n\in \Nn$.
The above series, moreover, is absolutely convergent on any relatively compact open set, and so are all its partial derivatives.
\end{prop}

\section{Real Du Val symplectic resolutions}\label{sec:DuVal}

\noindent In this section, we mainly recall from \cite{sym} and adapt to the real case the construction of the symplectic resolution of the Du Val singularities $ A_{2k} $ and $ D_{2k+1} $. 

First, we define a Poisson structure on the varieties $ A_{2k} $ and $ D_{2k+1} $.
For any polynomial function $ F(x,y,z)$ on $\Rr^3 $, a Poisson bracket on $\Rr^3 $ is given by \cite{Pichereau2}:
\begin{equation}
 \label{eq:Pichereau} \{x,y\} = \frac{\partial F}{\partial z}, \{y,z\} = \frac{\partial F}{\partial x}, \{z,x\} = \frac{\partial F}{\partial y} .\end{equation}
This Poisson bracket admits $F$ as a Casimir function, and therefore descends to the quotient $\tfrac{\mathbb R[x,y,z]}{(F)} $. Unlike in the complex case, the latter algebra may not be easily identified with polynomial functions on the set $F=0$ (for instance, for $ F= x^2+y^2+z^2$, the zero locus is a point). 
It only makes sense when the quotient $\tfrac{\mathbb R[x,y,z]}{(F)} $ injects itself into the algebra of real valued functions on $\{F =0\} $. This is the case, in particular, for $F= x^2 + y^2 - z^{2k}$. We denote by $ A_{2k}$ the set
$$A_{2k} =  \{(x,y,z) \in \Rr^3 \,  \mid \,  x^2 + y^2 = z^{2k}\} .$$
Formula \eqref{eq:Pichereau} equips functions on $A_{2k} $ with the following Poisson structure:
\begin{equation}\label{hichem2}
\left\{x,y\right\}=-2kz^{2k-1},\quad \left\{y,z\right\}=2x,\quad \left\{z,x\right\}=2y.
\end{equation}

Let us study the successive blowing up at the origin (the only singular point) of $A_{2k} $ and the behavior of the above Poisson structure.

Let us first describe the minimal resolution of $ A_{2k}$.
\begin{enumerate}
    \item The blowup $\varphi:Z\to A_{2k}$ at the origin $(0,0,0) \in A_{2k}$ is covered by $3$ open charts. 
     \begin{enumerate}
         \item In two of them, the strict transform  of $A_{2k}$ is smooth. 
         \item In the last one, the $z$-chart, with coordinates $u, v, z$, the morphism is given by
$\varphi_1:  Z_1 \to \Rr^3 $ defined by $x = uz, y = vz, z = z$, and the inverse image of $A_{2k}$
is defined by $u^2z^2+v^2z^2=z^{2k}$ so that  the strict transform is $ u^2+v^2 =z^{2(k-1)} $.
As a consequence, the strict transform  on this chart is isomorphic to~$ A_{2k-2} $.
     \end{enumerate}
    \item We then blow-up $ A_{2k-2}$, and repeat the procedure
    \item[]
    \item[$\vdots$]
    \item[]
    \item[ (k)]  Therefore, after $k$-successive blowups at $(0,0,0)$: $\varphi_k:  Z_k \to Z_{k-1}$, the 
strict transform is given by $x^2+y^2=1$ with $z$ a free variable. It is therefore smooth. 
\end{enumerate}
At this point, we obtain a smooth affino-projective variety, called the \emph{(real) Du Val resolution of $A_{2k} $}. 

Let us describe the behaviour of the Poisson structure under the successive blow-up leading to the Du Val resolution. 
\begin{enumerate}
    \item The Poisson structure \eqref{hichem2} pulls-back to a Poisson structure on the strict transform of $ A_{2k}$ under the first blow-up $ A_{2k}$.
    \begin{enumerate}
        \item In the two first affine charts, the pulled-back Poisson structure is symplectic.
        \item In the third one, (where the strict transform is $\simeq A_{2k-2}$), the pulled-back Poisson structure is of the form \eqref{hichem2} with $k$ being replaced by $ k-1$.
    \end{enumerate}
     \item The same applies then to the blow-up of $ A_{2k-2}$
    \item[]
    \item[$\vdots$]
    \item[]
    \item[ (k)] Therefore, after $k$-successive blowups at $(0,0,0)$, the pulled-back Poisson structure, which is still well-defined and has no singularity, is given by
\begin{equation}\label{hic23}
\left\{x,y\right\}=0,\quad \left\{y,z\right\}=2x,\quad \left\{z,x\right\}=2y.
\end{equation}
which is symplectic.
\end{enumerate}



We have therefore obtained by successive blowing a symplectic resolution of $A_{2k}$, which we denote by $(Z_k,\pi_k)$. In conclusion:

\begin{prop}
\label{prop:sympresolA2k}
For all $k \geq 2$:
\begin{enumerate}
    \item 
The map $\varphi: Z_k \to A_{2k} $ is onto, and restricts to a diffeomorphism $ Z_k \backslash \varphi^{-1}(\{0\}) \simeq A_{2k} \backslash \{ 0\} $
    \item 
The pull-back of the Poisson structure \eqref{hichem2} to the Du Val resolution of $ \varphi : Z_k \to A_{2k}$ is a symplectic structure.
\end{enumerate}
Equivalently, $(Z_k,\varphi,\pi_k)$ is a symplectic resolution of $ A_{2k}$. 
 \end{prop}

\noindent The next point is easy to verify and will be important. Notice that, now, it is only true on $\Rr$, and not valid on $\Cc$.

\begin{prop}\label{coco3}
The pull back of the rational functions $\frac{x}{z^{k}}$ and $\frac{y}{z^{k}}$ by the symplectic Du Val resolution $Z_k\rightarrow A_{2k}$ are smooth functions on $Z_k$.
\end{prop}

\begin{proof}
 The unique point where the rational function $\frac{x}{z^{k}}$ is not defined is  $x=y=z=0$.
Let us see the behaviour under the successive blowups that lead to the Du Val resolution:
\begin{enumerate}
    \item 
For the blowing-up of the singularity $ A_{2k} $ at the origin, the following phenomenon occurs:
    \begin{enumerate}
        \item in the two first affine charts (the $x$-chart and the $y$-chart), the pull-back is a smooth function, because the pull-back of $z$ can not be equal to $0$,
        \item in the last one, the restriction to the strict transform (which is isomorphic to $A_{2k-2} $) of the pull back of $\frac{x}{z^{k}}$ is the function $\frac{x}{z^{k-1}}$.
    \end{enumerate}
    \item[] We are therefore brought back to the previous function, upon changing $ k $ by $ k-1 $. We can then repeat the procedure.
    \item[$\vdots$]
    \item[]
    \item[ (k)] At the last step, the strict transform is given by $ x^{2}+y^{2}-1=0$  with $z$ being free, and the  back pull of the function $x/z^k $ is simply the function $x$, which is smooth.
\end{enumerate}
This completes the proof.
\end{proof}
%

\noindent Let us define two sequences $(x_n)_{n \geq 1}$ and $ (y_n)_{n \geq 1}$ of rational functions on $A_{2k} \subset \Rr^3$ by $x_1(x,y,z):=x$, $y_1(x,y,z):=y$ for all $(x,y,z) \in A_{2k}$ and 
\begin{equation}
\label{eq:xnyn}
  \left\{\begin{array}{rcl}
              x_{n}=\frac{x_{1}}{z^{k}}x_{n-1}-\frac{y_{1}}{z^{k}}y_{n-1}\\
						  y_{n}=\frac{x_{1}}{z^{k}}y_{n-1} + \frac{y_{1}}{z^{k}}x_{n-1}.
 \end{array}\right.
\end{equation}
From Proposition \ref{coco3}, it follows that 

\begin{coro}\label{coronYn}
For all $n\in \Nn $ the pull back $X_n := \varphi^* x_n$ and $Y_n = \varphi^* y_n $ of the functions $x_n,y_n$ defined in (\ref{eq:xnyn}) to the Du Val resolution $ \varphi: Z_k \to A_{2k}$ are smooth functions on $Z_k$.
\end{coro}

\noindent By same procedure, we equip the variety defined by 
$$  D_{2k+1} := \left\{ (x, y, z) \in \Rr^3 \middle| zx^2 + y^2 = z^{2k+1} =0  \right\}$$ 
with the Poisson structure defined as in \eqref{eq:Pichereau} with the help of the function $ zx^2 + y^2 - z^{2k+1}$:
 \begin{equation}\label{PoissonDk} \{x,y\}= -(2k+1)z^{2k}+x^2 , \{y,z\}= 2zx, \{z,x\}= 2y .\end{equation}
Let us now describe a symplectic resolution of the latter. The procedure is at first very similar.
\begin{enumerate}
    \item 
The strict transform of the blowing up at zero is covered by the three natural charts. 
\begin{enumerate}
    \item In the $x$-chart the strict transform is the variety $  zx-z^{2k+1}x^{2k-1} = - y^2$, whose unique singular point is $(0,0,0) $. 
    \begin{enumerate}
    \item The strict transform of the pull-back at the origin of the latter variety is smooth, and the pull-back of the Poisson structure \eqref{PoissonDk} is symplectic.
    \end{enumerate}
    \item  In the $y$-chart the strict transform is smooth and the pull-back of the Poisson structure \eqref{PoissonDk} is symplectic.
     \item  In the $z$-chart the strict transform is $ D_{2(k-1)}$, and the pull-back of the Poisson structure \eqref{PoissonDk} is the Poisson structure  \eqref{PoissonDk}, with $k$ being replaced by $ k-1$.
\end{enumerate}
   We are therefore brought back to the previous problem in changing $ k $ by $ k-1 $ and then repeat the operation.
    \item[$\vdots$]
    \item[]
    \item[ (k)] At the $k^{th}$-step, the situation is quite different from the one of $ A_{2k}$. The strict transform is the variety:
    \begin{equation*}
            zx^{2} + y^{2}-z=0
\end{equation*}
     which admits two singular points: $(1,0,0)$ and $(- 1,0,0)$.
     The pull-back of the Poisson structure \ref{PoissonDk} is given (as in \eqref{eq:Pichereau}) by 
       \begin{equation}\label{PoissonD0} \{x,y\}= x^2 -1, \{y,z\}= 2x z, \{z,x\}= 2y .\end{equation}
     \begin{enumerate}
         \item The blow up at $(1,0,0)$  admits a singular point.  Blowing-up at this point, we obtain at last a smooth variety, and the pull-back of the Poisson structure \eqref{PoissonD0} is symplectic.
         \item The blow up at $(-1,0,0)$  admits a singular point.  Blowing-up at this point, we obtain at last a smooth variety, and the pull-back of the Poisson structure \eqref{PoissonD0} is symplectic.
         We denote by $(Z_{k+2}, \pi_{k+2}) $ this symplectic manifold.
     \end{enumerate}
\end{enumerate}



Let $D_{2k+1}^+=D_{2k+1} \cap \{z>0\}$ and  $Z_{k+2}^+ :=\varphi^{-1}(\{z>0\}) \subset Z_{k+2} $.  $Z_{k+2}^+$ is  manifold (with boundary). We have:

\begin{prop}
\label{prop:sympresolA2k1}
For all $k \geq 2$:
\begin{enumerate}
    \item 
The map $\varphi: Z_{k+2}^+\to D_{2k+1}^+ $ is onto, and restricts to a diffeomorphism $ Z_{k+2}^+ \backslash \varphi^{-1}(\{0\}) \simeq D_{2k+1}^+ \backslash \{ 0\} $.
    \item 
The pull-back of the Poisson structure \eqref{PoissonDk} to the Du Val resolution of $ \varphi : Z_{k+2}^{+} \to D_{2k+1}^+$ is a symplectic structure.
\end{enumerate}
Equivalently, $(Z_{k+2}^+,\varphi,\pi_k)$ is a symplectic resolution of $ D_{2k+1}^+$. 
 \end{prop}

\begin{prop}
The pull back $\frac{x}{z^{k+\nicefrac{1}{2}}}$ and $\frac{y}{z^{k+\nicefrac{1}{2}}}$ by the map $Z_{k+2}^{+}\rightarrow D_{2k+1}^+$ are smooth functions on $Z_{k+2}$.
\end{prop}

\begin{proof}[Proof]
The proof is similar to that of the Proposition \ref{coco3}. We leave it to the reader.
\end{proof}

\noindent We shall consider the following functions, similar to those in (\ref{eq:xnyn}):

\begin{equation}
\label{eq:xnyn1}
  \left\{\begin{array}{rcl}
              x_{n}=\frac{x_{1}}{z^{k+\nicefrac{1}{2}}}x_{n-1}-\frac{y_{1}}{z^{k+\nicefrac{1}{2}}}y_{n-1}\\
						  y_{n}=\frac{x_{1}}{z^{k+\nicefrac{1}{2}}}y_{n-1} + \frac{y_{1}}{z^{k+\nicefrac{1}{2}}}x_{n-1}.
 \end{array}\right.
\end{equation}
It follows from all of the above that:

\begin{coro}\label{coronYn23}
The pull back $X_n := \varphi^* x_n$ and $Y_n := \varphi^* y_n $ of functions $x_n,y_n$ defined in (\ref{eq:xnyn1}), by the Du Val resolution $ \varphi : Z_{k+2} \to D_{2k+1}$ are polynomial functions on $Z_{k+2}$.
\end{coro}

\section{Quotient Poisson structure}\label{sec:Quotient2}

\noindent We now introduce the Poisson bracket on the quotient space $\Rr^2\hspace*{-1.5pt} / \hspace*{-1.5pt}G$ (resp $\Rr^2\hspace*{-1.5pt} / \hspace*{-1.5pt}G'$), with $G$ and $G'$ acting on $M = \mathbb R^2 $ as in Section \ref{sec:quotient}.
We introduce for $ k \geq 2 $ the algebra of functions denoted by $C^{k,\infty}(\Rr^2)$ which are:
\begin{enumerate}
    \item[a)] differentiable  $\floor{k/2}$-times on $\mathbb R^2 $,
    \item[b)] and smooth outside the straight line $q=0 $.
\end{enumerate}

\noindent
 We now  equip $M=\Rr^2$ with the following Poisson bracket
\begin{equation} \label{eq:Poisson}\{p,q\}=\frac{q}{2\pi}. \end{equation}

\noindent
The following lemmas are easily verified.
\begin{lem} \label{lem:stable} For every $k\in\left\{2,3,\ldots,+\infty\right\}$, the algebra $C^{k,\infty}(\Rr^{2})$ is stable under the Poisson bracket \eqref{eq:Poisson}, i.e is a Poisson algbera. 
\end{lem}

\noindent
and

\begin{lem}
\label{lem:invariant}
The Poisson bracket \eqref{eq:Poisson} is invariant under the actions of $G$ and $G'$.
\end{lem}

\noindent 
In view of Lemmas \ref{lem:stable} and \ref{lem:invariant}, the following algebras come equipped with the following induced Poisson brackets:
 $C^\infty(M\hspace*{-1.5pt} / \hspace*{-1.5pt}G) := C^\infty(M)^G$, $C^\infty(M\hspace*{-1.5pt} / \hspace*{-1.5pt}G'):=C^\infty(M)^{G'}$, $ C^{k,\infty}(M\hspace*{-1.5pt} / \hspace*{-1.5pt}G):= C^{k,\infty}(M)^G$, and $C^{k,\infty}(M)^{G'} $ (see Section \ref{sec:quotient} for notations).

\begin{rema}
\normalfont
The symplectic Poisson bracket $ \left\{p,q\right\}=1 $ is also invariant under the actions of $G$ and $G'$, but does not satisfy Lemma \ref{lem:stable}.
Also, since the singular locus of both group actions (\emph{i.e.} the straight line $q=0$)  projects to a single  point in the quotient space, it makes more sense to choose a Poisson structure on $\mathbb R^2 $ that vanishes on this straight line. 
\end{rema}

From now on, we  deal with the $G$-action only. Results for the $G'$-action will be stated at the end of the section.

\noindent Let $\mathcal{A}$ be the subalgebra of $C^{k,\infty}(\Rr^2\hspace*{-1.5pt} / \hspace*{-1.5pt}G)$ generated by $$\left(x(p,q):=q^{k}\cos\left(2\pi\frac{p}{q}\right), \, \, y(p,q):=q^{k}\sin\left(2\pi\frac{p}{q}\right), \, \, z(p,q):=q\right).$$ 

\begin{prop}
The algebra $\mathcal A $ is stable under the Poisson bracket \eqref{eq:Poisson}.
\end{prop}
\begin{proof}
 A direct computation with this bracket gives
\begin{equation}
\label{eq:bracketsinA}
\left\{\begin{array}{rcll}
\{x,y\}&=&-2k \, z^{2k-1}  &\in \mathcal{A}\\
          \{y,z\}&=&2 \, x& \in \mathcal{A}\\
           \{z,x\}&=&2 \, y & \in \mathcal{A}.
 \end{array}\right.\end{equation}
 This proves the result.
\end{proof}





Let us consider the map $ \mathbb R^2 \to \mathbb R^3$
defined by
 $$ \phi_k \colon \big(p,q\big) \mapsto \big(x(p,q),y(p,q),z(p,q)\big) .$$
 Since $x,y,z$ are $G$-invariant functions, and are equal to $0$ on the straight line $ q=0$, this map goes to the quotient to define a map
 $\underline{\phi_k} \colon \mathbb R^2 \hspace*{-1.5pt} / \hspace*{-1.5pt} G \to \mathbb R^3$.
Since $x^2 + y ^2 = z^{2k} $, this map takes values in the Du Val singular surface $A_{2k} $ (see Section \ref{sec:quotient}).

\noindent  
\begin{lem}
\label{afair}
The map $ \underline{\phi_k}  \colon \mathbb R^2\hspace*{-1.5pt} / \hspace*{-1.5pt}G \to A_{2k}$
 is one-to-one.
Its restriction to $ (\mathbb R^* \times \mathbb R) \hspace*{-1.5pt} / \hspace*{-1.5pt} G \simeq A_{2k} \backslash \{0\} $ is a diffeomorphism. 
 \end{lem}
 \begin{proof}
The relation \begin{equation}\label{presquefin}
\begin{split}
x^{2}+y^{2}&=\left(q^{k}\cos\left(2\pi\frac{p}{q}\right)\right)^{2}+\left(q^{k}\sin\left(2\pi \frac{p}{q}\right)\right)^{2}\\
           &=q^{2k} =z^{2k}.
\end{split}
\end{equation}
proves that $\phi_{k} $ is valued in $A_{2k} $. Hence so is $\underline{\phi_{k}}$.

Injectivity follows from the following two points:
\begin{itemize}
\item For any two $(p,q),(p',q')$ with $q,q'\ne0$,  $x(p,q)=y(p,q)$ if and only if $q=q'$ and $p=p'+nq$ for some $n\in\Nn$. 
\item $x(p,q)=y(p,q)=z(p,q)=0 $ implies $q=0$
\end{itemize}
Hence $\underline{\phi_k} (p,q) = \underline{\phi_k}(p',q')$ if and only if $(p,q) \sim (p',q') $.

\noindent
 Surjectivity follows easily from the surjectivity of $u\to(\cos{u},\sin{u})$ from $\Rr$~to~$S^1$.  
\end{proof}
 

%




\noindent
Both $\mathbb R^2\hspace*{-1.5pt} / \hspace*{-1.5pt}G$ and $A_{2k}$ come equipped with Poisson structures. We would like to state that $\underline{\phi_k} : \mathbb R^2\hspace*{-1.5pt} / \hspace*{-1.5pt}G \simeq A_{2k} $ is a Poisson isomorphism. 
But we have to be careful because there are several algebras of functions on $\mathbb R^2\hspace*{-1.5pt} / \hspace*{-1.5pt}G$. The algebra that allows to make this statement is $\mathcal A $, as we now state
(here $\mathcal{F}(A_{2k})= \frac{\mathbb R[x,y,z]}{x^2+y^2-z^{2k}}$ stands for the algebra of polynomial functions on $A_{2k} $):

\begin{prop}\label{jti}
The pull-back map of $\underline{\phi_{k}}:{\mathbb R}^2\hspace*{-1.5pt} / \hspace*{-1.5pt}G \to A_{2k}$ is a Poisson morphism
 $$C^\infty(\mathbb R^3)  \to C^{\floor{k/2}} (\mathbb R^2\hspace*{-1.5pt} / \hspace*{-1.5pt}G).$$
 It restricts to a Poisson
algebra isomorphism 
$\underline{\phi_k}^* :\mathcal{F}(A_{2k})\simeq \mathcal A$. 
\end{prop}

\begin{proof}[Proof]
By Equation \eqref{presquefin}, the functions $x,y,z \in C^{k,\infty}(\mathbb R^2) $ satisfy  $x^2+y^2 = z^{2k} $.
Since there is no other relations between them, save for multiples of this relation, and since the generators $x,y,z $ of $\mathcal{F}(A_{2k}) $ are also subject to the relation  $x^2+y^2 = z^{2k} $, the algebra morphism $\underline{\phi_k}^*:\mathcal{F}(A_{2k})\hookrightarrow \mathcal A$ is an algebra isomorphism.
It is still necessary to show that it is a Poisson morphism. 

\noindent
A comparison of the Poisson bracket on $\mathcal{F}(A_{2k}) $ defined in \eqref{hichem2}
with the Poisson brackets of $x,y,z \in C^{k,\infty}(M)$ as computed in (\ref{eq:bracketsinA})
shows that $\phi_k: \mathbb R^2  \to \mathbb R^3  $ is a Poisson map.
In view of the following commutative diagram
\begin{equation}\label{prsq1}
\xymatrix{ \Rr^{2} \ar[r]^{\phi_k} \ar@{->>}[d]^{P} & \Rr^{3} \\ \Rr^{2}\hspace*{-1.5pt} / \hspace*{-1.5pt}G \ar[r]^{\underline{\phi_k}}& A_{2k} \ar@{^{(}->}[u]^{i} }
\hbox{ \begin{tabular}{c} \\  \\ \\  hence  \\ \end{tabular}}
\xymatrix{
    \mathbb R[x,y,z] \ar[r]^{{\phi_k}^{*}} \ar@{->>}[d]^{i^{*}} & C^{\floor{k/2}}(\Rr^{2}) \\
    \mathcal F (A_{2k}) \ar[r]^{\underline{\phi_k}^{*}} & C^{\floor{k/2}}(\Rr^{2}\hspace*{-1.5pt} / \hspace*{-1.5pt}G) \ar@{^{(}->}[u]^{P^{*}}}
\end{equation}
this implies that $\underline{\phi_k}^* $ is a Poisson morphism.

\end{proof}

\noindent For the quotient space ${\mathbb R}^2\hspace*{-1.5pt} / \hspace*{-1.5pt}G'$, we consider the subalgebra $\mathcal{A}'\subset C^{k}(\Rr^2\hspace*{-1.5pt} / \hspace*{-1.5pt}G')$ generated by 
\begin{equation*}
\left(x=q^{2k}\cos\left(2\pi\frac{p}{q}\right),  \,  y=q^{2k+1}\sin\left(2\pi\frac{p}{q}\right), \,  z=q^2\right), \, \mbox{for any integer } k\geq2.
\end{equation*} 

\begin{prop}
The algebra $\mathcal{A}'$ is stable under the Poisson bracket \eqref{eq:Poisson}.
\end{prop}

\begin{proof}[Proof]
A direct computation with this bracket gives
\begin{equation*}
\left\{\begin{array}{rcll}
\{x,y\}&=&-(2k+1) z^{2k}+x^2 &\in \mathcal{A}'\\
          \{y,z\}&=&2z x&\in \mathcal{A}'\\
           \{z,x\}&=&2y &\in \mathcal{A}'.
 \end{array}\right.\end{equation*}
 This proves the result.
\end{proof}

\noindent Let us consider the map $ \mathbb R^2 \to \mathbb R^3$
defined by
 $$ \phi_k \colon \big(p,q\big) \mapsto \big(x(p,q),y(p,q),z(p,q)\big) .$$
 Since $x,y,z$ are $G'$-invariant functions, and are equal to $0$ on the straight line $ q=0$, this map goes to the quotient to define a map 
 $\underline{\phi_k} \colon \mathbb R^2\hspace*{-1.5pt} / \hspace*{-1.5pt}G' \to \mathbb R^3$.
Since $zx^2 + y ^2 = z^{2k+1} $, this map takes values in the Du Val singular surface $D_{2k+1} $ (see Section \ref{sec:quotient}). More precisely, it takes values in $D_{2k+1}^+ = D_{2k+1} \cap \{z \geq 0\} $

\noindent  
\begin{lem}
\label{afair2}
The map $ \bar{\phi_k}  \colon \mathbb R^2\hspace*{-1.5pt} / \hspace*{-1.5pt}G' \to D_{2k+1}^+$
 is one-to-one
 \end{lem}
 \begin{proof}
The relation \begin{equation}\label{presquefin1}
\begin{split}
zx^{2}+y^{2}&=\left(q^{2k+1}\cos\left(2\pi\frac{p}{q}\right)\right)^{2}+\left(q^{2k+1}\sin\left(2\pi k\frac{p}{q}\right)\right)^{2}\\
           &=q^{4k+2} =z^{2k+1}
\end{split}
\end{equation}
proves that $\phi_{k} $ is valued in $D_{2k+1} $. Since $ q^2 \geq 0$, it takes indeed values in $D_{2k+1}^+$. Hence so is $\bar{\phi_k}$. It remains to prove that $\bar{\phi_k}$ is bijective. The proof is identical to that of the Lemma \ref{afair}.
\end{proof}

\noindent As for $\Rr^2\hspace*{-1.5pt} / \hspace*{-1.5pt}G$ and $A_{2k} $, since $\Rr^2\hspace*{-1.5pt} / \hspace*{-1.5pt}G'$ and $D_{2k+1}^+$ are in bijection, and since they are also Poisson sets, we would like to assert that $\bar{\phi_k}$ is a Poisson isomorphism. As for $\Rr^2\hspace*{-1.5pt} / \hspace*{-1.5pt}G$ and $A_{2k} $, the precise statement will use the subalgebra $\mathcal{A}'\subset C^{k}(\Rr^2\hspace*{-1.5pt} / \hspace*{-1.5pt}G')$.

\begin{prop}\label{jti1}
The pull-back map of $\bar{\phi_{k}}:{\mathbb R}^2\hspace*{-1.5pt} / \hspace*{-1.5pt}G' \to D_{2k+1}^+$ is a Poisson morphism:
 $$C^\infty(\mathbb R^3) \to C^{\floor{k\hspace*{-1.5pt} / \hspace*{-1.5pt}2}} (\mathbb R^2\hspace*{-1.5pt} / \hspace*{-1.5pt}G). $$
It restrict to a Poisson algebra isomorphism 
$\bar{\phi_k}^* :\mathcal{F}(D_{2k+1})\simeq \mathcal A'$. 
\end{prop}

\noindent The proof of this Proposition is identical to that of the Proposition \ref{jti}.

\section{Symplectic resolutions of the quotients $\mathbb R^2\hspace*{-1.5pt} / \hspace*{-1.5pt}G $ and $\mathbb R^2\hspace*{-1.5pt} / \hspace*{-1.5pt}G' $}\label{sec:resolution}

\noindent We first deal with the case of $G$. We use the notations of the previous section.

\noindent The map $ \underline{\phi_k}: \Rr^2\hspace*{-1.5pt} / \hspace*{-1.5pt}G \to A_{2k}$ is bijective by Lemma \ref{afair}. 

\begin{defi}\label{def:nota}
We denote by $\psi_k :A_{2k} \to \Rr^2\hspace*{-1.5pt} / \hspace*{-1.5pt}G $ its inverse map of $ \underline{\phi_k}: \Rr^2\hspace*{-1.5pt} / \hspace*{-1.5pt}G \to A_{2k}$,
and by $\varphi:Z_{k} \to A_{2k}$ the usual symplectic resolution of $A_{2k}$ (see Section \ref{sec:DuVal}). 
\end{defi}

Corollary \ref{coronYn} will now be used to show the very surprising following result:

\begin{prop}\label{jiti}
For every $k \in \mathbb N$, the map $\psi_k  \, \circ \, \varphi$ is a smooth map from $(Z_{k},\{\cdot, \cdot\}_{Z_{k}})$ to $(\Rr^2\hspace*{-1.5pt} / \hspace*{-1.5pt}G, \{\cdot,\cdot\}) $, \emph{i.e.} for every smooth function $F \in C^\infty(\Rr ^2\hspace*{-1.5pt} / \hspace*{-1.5pt}G)$, its pull-back $(\psi_k\circ\varphi)^*F$ is a smooth function on $Z_k$.
\end{prop}

\begin{proof}[Proof]
 Let $F \in C^\infty(\Rr ^2\hspace*{-1.5pt} / \hspace*{-1.5pt}G)$. By Proposition \ref{refr}, we have 
$$ F(p,q) = \alpha_{0}(q)+\sum_{n\geq 1}\alpha_{n}(q)\cos\left(2n\pi\frac{p}{q}\right)+\sum_{n\geq1}\beta_{n}(q)\sin\left(2n\pi\frac{p}{q}\right) $$
for some smooth functions $(\alpha_{n})_{n\geq 1}$ and $(\beta_{n})_{n\geq 1}$ on $ \mathbb R$ that vanish at $0$ together with all their derivatives, and some smooth function $\alpha_{0}$ (all unique).

\noindent We wish to define a function $ H $ on $ Z_{k} $ by
\begin{equation}\label{eq:defH} H := \alpha_{0}(Z)+\sum_{n\geq 1} \frac{\alpha_{n}(Z)}{Z^k}   X_n +\sum_{n\geq1}
\frac{\beta_{n}(Z)}{Z^k}  Y_n\end{equation}
where the functions $ X_n,Y_n \in C^\infty(Z_k)$ are as in Corollary \ref{coronYn} and $ Z= \varphi^* z=q$.
This needs some justification.
\begin{enumerate}
\item All terms of both series defining $H $ (see \eqref{eq:defH}) are smooth:
\begin{enumerate}
\item Corollary \ref{coronYn} states that $X_n,Y_n $ are smooth functions for all $n \geq 1$. 
\item $Z = \varphi^* z$ is also a smooth function.
\item Last, since the functions $\alpha_n, \beta_n $ vanish at $0$ together with all their derivatives, the quotients 
$$ \alpha_n(z) / z^n, \beta_n (z) / z^n $$
are smooth functions on $\mathbb R $. In view of (b),  $\alpha_n(Z) / Z^n, \beta_n (Z) / Z^n $ are also smooth functions.
\end{enumerate}
\item Both series are convergent, together with all their partial derivatives.
We will come back to this point at the end the proof.
\end{enumerate}

Now that the existence and the smoothness of the function $H$ is established, let us check that 
\begin{equation} \label{eq:tirearriere} (\psi_k \circ \varphi)^* F = H  .\end{equation} 
It suffices to check that the pull-back through $(\psi_k \circ \varphi)^*$ of each one of the terms of the series defining $F$ are the terms of both series defining $H$.
Consider the following smooth functions on $\mathbb R^3 $
$$f_n(x,y,z):=\frac{\alpha_{n}(z)}{z^k} x_n  \hbox{ and } g_n(x,y,z) := \frac{\beta_{n}(z)}{z^k} y_n  $$
where the functions $x_n,y_n$ are as in (\ref{eq:xnyn}).
We have
$$ \left\{ \begin{array}{rcll}
 \varphi^* f_n &= &    \frac{\alpha_{n}(Z)}{Z^n} X_n & \hbox{by definition of $X_n$ (see Cor. \ref{coronYn}) } \\
  \underline{\phi_k}^* f_n & = & \alpha_{n}(q)\cos\left(2n\pi\frac{p}{q}\right)  & \hbox{by definition of $\underline{\phi_k} $} 
\end{array} \right. $$
and 
$$ \left\{ \begin{array}{rcll}
 \varphi^* g_n &= &   \frac{\beta_{n}(Z)}{Z^n} Y_n & \hbox{by definition of $Y_n$ (see Cor. \ref{coronYn}) } \\
  \underline{\phi_k}^* g_n & = & \beta_{n}(q)\sin \left(2n\pi\frac{p}{q}\right)  & \hbox{by definition of $\underline{\phi_k} $} 
\end{array} \right. $$
Hence
 \begin{align*}
 \begin{array}{rcl}
     \varphi^* \circ \psi_k^* \left( \alpha_{n}(q)\cos \left(2n\pi\frac{p}{q}\right)\right) &=&\alpha_{n}(Z) X_n  \\ \varphi^* \circ \psi_k^* \left(  \beta_{n}(q)\, \sin \left(2n\pi\frac{p}{q}\right)\right) &=& \beta_n (Z) \, Y_n\end{array}
 \end{align*} 
 \noindent Summing up these identities, we obtain \eqref{eq:tirearriere}.



\noindent To complete the proof, we need to prove uniform convergence of the series \eqref{eq:defH} defining $ H $, and uniform convergence of all the series obtained by applying a given partial differential operator to all their terms.  

 \noindent
Let $C \subset Z_k $ be a compact neighborhood of $\phi^{-1}(0) $. 
Pulling-back the relation
 $$ x_n^2+y_n^2=z^{2k} $$
 we obtain that  
 \begin{equation}\label{eqXn} X_n^2 + Y_n^2 = Z^{2k} \end{equation}
 so that $X_n,Y_n $ is bounded on $C$ by some constant $k = {\mathrm{max}}_C |Z| $ that does not depend on $n$.

 \noindent
Now let us consider $\varphi_k^{-1} ( \phi(C)) \subset \mathbb R^2  $. 
This is a neighborhood of the horizontal line in $\mathbb R^2 $, contained into a band $\mathcal B $ of the form $ - M <q< M $ for some $M \in \mathbb R $. 
In view of Lemma \ref{lem:alphank}, there exists a constant $K$ such that 
$| \alpha_n(q)/q^k |< K / n^k $ and $| \beta_n(q)/q^k |< K / n^k $ for all $ -M<q<M $, so that 
$$ \left| \frac{\alpha_n(Z)}{Z^k} \right| < \frac{K}{n^k} \hbox{ and } \left| \frac{\alpha_n(Z)}{Z^k} \right| < \frac{K}{n^k}  $$
uniformly on $C$. This proves the uniform convergence on $C$ of the series \eqref{eq:defH} defining~$H$.

 \noindent
To prove that $H$ is smooth, it suffices to prove that the series that define $D(H)$ converge for every partial derivative operator $D$ on $Z$.

 \noindent
 For every partial derivative operator $D$ of order $d$, we derive from the recursion relation \eqref{eq:xnyn} that
  $$ \begin{pmatrix} D(X_{n+1})\\ D (Y_{n+1}) \end{pmatrix} = \begin{pmatrix} X_1& -Y_1\\ Y_1 & X_1 \end{pmatrix} \begin{pmatrix} D (X_n) \\ D(Y_n) \end{pmatrix}  +  \begin{pmatrix} L(X_n,Y_n) \\ M(X_n,Y_n) \end{pmatrix} $$
where $L(X,Y),M(X,Y)$ depend linearly on a finite number of partial derivatives of order $\leq d-1 $ of $X,Y $. 
Hence 
$$ U_{n}^D :=  (  D(X_{n}))^2 + ( D(Y_{n}))^2$$
satisfies 
\begin{equation}\label{eq:Un} U_{n+1}^D = U_n^D + S^D(X_n,Y_n) \end{equation}
 where $S^D(X,Y)$ depend linearly on a finite number of partial derivatives of order $\leq d-1 $ of $X,Y $.

 \noindent
 We claim that for any partial derivative operator $D$ of degree $d$, there exists constants $A,B $ such that 
 $ U_n^D \leq  A +  B n^D $.
 For $d=0$, Equation \eqref{eqXn} implies the desired assumption. 
 If this is true for all $D$ of degree $\leq d-1 $, then in
 Equation \eqref{eq:Un}, there exists $A,B$ such that
  $$ |S^D(X_n,Y_n)| \leq A + B n^{d-1} $$
Equation \eqref{eq:Un}   
implies then that there exists constants $A',B' $ such that 
 $ U_n^D \leq  A' +  B' n^d $.
 
 \noindent
The convergence to zero of the $d$-first derivatives of $|\alpha_n(Z)/Z^k|, |\beta_n(Z)/Z^k| $ can be
 made smaller that $\frac{K}{n^r} $ for an arbitrary $r$, uniformly on $\phi^{-1}(C) $. This proves the uniform convergence of  the series obtained by applying  $D$ to all the terms in the series \eqref{eq:defH}. This proves the smoothness of its limits, and therefore the smoothness of $H$.


\end{proof}

\noindent 
We recapitulate in the following diagram the most important maps used so far
\begin{equation}
\xymatrix{ \Rr^{2} \ar[r]^{\phi_k} \ar@{->>}[d]^{P} & \Rr^{3}  & Z_{k} \ar[dl]^{\varphi} \\ \Rr^{2}\hspace*{-1.5pt} / \hspace*{-1.5pt}G \ar@/^1.0pc/
[r]^{\underline{\phi_k}}& A_{2k} \ar@/^1.0pc/[l]^{\underline{\psi_k}}  \ar@{^{(}->}[u]^{i} } 
\end{equation}
and we conclude with the main result of this article:

\begin{thm}\label{resultatp}
For every $k \in \mathbb N$, the map $\psi_k \circ \varphi: Z_{k} \rightarrow \Rr^{2}\hspace*{-1.5pt} / \hspace*{-1.5pt}G$ is a symplectic resolution of  $(\Rr^{2}\hspace*{-1.5pt} / \hspace*{-1.5pt}G,\{\cdot,\cdot \}) $.
\end{thm}

\begin{proof}[Proof]
Let us check that all four conditions in Definition \ref{def:Lassoued} are satisfied.
Condition \emph{(i)} has been established in Proposition \ref{jti}.
Condition \emph{(ii)} follows from Proposition \ref{jti} and the first item in Proposition \ref{prop:sympresolA2k}.
Conditions \emph{(iii)} and  \emph{(iiii)} hold in view of Lemma \ref{afair} and the second item in Proposition~\ref{prop:sympresolA2k}.
%
%
%
\end{proof}
\noindent

\noindent For the quotient space ${\mathbb R}^2\hspace*{-1.5pt} / \hspace*{-1.5pt}G'$, the map $ \bar{\phi_k}: \Rr^2\hspace*{-1.5pt} / \hspace*{-1.5pt}G' \to D_{2k+1}^+$ is bijective by Lemma \ref{afair2}. Let us call $\underline{\psi_k} :D^+_{2k+1} \to \Rr^2\hspace*{-1.5pt} / \hspace*{-1.5pt}G' $ its inverse map. We denote by $\varphi:Z_{k+2}^+ \to D_{2k+1}^+$ the usual symplectic resolution of $D_{2k+1}^+$, see Section \ref{sec:DuVal}.

\begin{prop}
For every $k \in \mathbb N$, the map $\underline{\psi_k} \circ \varphi$ is a smooth map from $(Z_{k+2}^+,\{\cdot, \cdot\}_{Z^+_{k+2}})$ to $(\Rr^2\hspace*{-1.5pt} / \hspace*{-1.5pt}G', \{\cdot,\cdot\}) $.
\end{prop}

\noindent 
\begin{proof}
The proof of this result is identical to that of Proposition \ref{jiti}. We simply use Proposition \ref{refr1} instead of Proposition \ref{refr}, and Corollary \ref{coronYn23} instead of Corollary~\ref{coronYn}.
\end{proof}

\begin{thm}
For every $k \in \mathbb N$, the map $\underline{\psi_k} \circ \varphi : Z_{k+2}^+ \rightarrow \Rr^{2}\hspace*{-1.5pt} / \hspace*{-1.5pt}G'$ is a smooth symplectic resolution of the quotient Poisson $(\Rr^{2}\hspace*{-1.5pt} / \hspace*{-1.5pt}G',\{\cdot,\cdot \} ) $.

\end{thm}

\begin{proof}[Proof]
Let us check that all four conditions in Definition \ref{def:Lassoued} are satisfied.
Condition \emph{(i)} has been established in Proposition \ref{jti1}.
Condition \emph{(ii)} follows from Proposition \ref{jti1} and the first item in Proposition \ref{prop:sympresolA2k1}.
Conditions \emph{(iii)} and \emph{(iiii)} hold in view of Lemma \ref{afair2} and the second item in Proposition~\ref{prop:sympresolA2k1}.
%
%
%
\end{proof}

\end{document}